\documentclass[12pt]{amsart}
\usepackage{mathrsfs}
\usepackage{}
\usepackage{amssymb, amstext, amscd, amsmath}
\usepackage{amsthm}
\usepackage{graphicx}
\usepackage{txfonts}
\usepackage{enumerate}
\newtheorem{thm}{Theorem}[section]

\newtheorem{prop}[thm]{Proposition}
\newtheorem{lem}[thm]{Lemma}

\theoremstyle{definition}
\newtheorem{rem}[thm]{Remark}

\newtheorem{claim}[thm]{Claim}
\numberwithin{equation}{section}

\begin{document}

\title[on keen weakly reducible Heegaard splittings ]
{on keen weakly reducible Heegaard splittings}

\author{Qiang E}
\address{Department of Mathematics, Dalian Maritime University, Dalian, P.R.China} \email{eqiang@dlmu.edu.cn}

\thanks{This research is supported by  grants
(No.11401069 and No.11671064) of NSFC and Fundamental Research Funds for the Central Universities (No. 3132015231)}



\begin{abstract} A Heegaard splitting which admits a unique pair of disjoint compression disks on distinct sides is said to be keen weakly reducible. This paper provides an construction of keen weakly reducible Heegaard splittings of arbitrary genus except 2. Furthermore, critical Heegaard splittings may yield if we change some conditions.

\end{abstract}
\maketitle

\section{Introduction}
Let $\emph{M}$ be an oriented compact 3-manifold. If there exists a
closed surface $\emph{S}$ which cuts $\emph{M}$ into two
compression bodies $\emph{V}$ and $\emph{W}$, such that
$S=\partial_{+}V=\partial_{+}W$, then $M=V\cup_{S}W$ is called a Heegaard spitting of $M$ and
$\emph{S}$ is called a Heegaard surface of $\emph{M}$. It is well known that every compact connected orientable 3-manifold admits a Heegaard splitting.

A Heegaard splitting $M=V\cup_{S}$W is said to be reducible if there are two essential disks $D\subset V$ and
$E\subset W$ such that $\partial D=\partial E$. In other words, there exists a 2-sphere which intersects $S$ in an essential curve; Otherwise, $M=V\cup_{S}W$ is said to be irreducible.

A Heegaard splitting $M=V\cup_{S}W$ is said to
be weakly reducible\cite{CG} if there are two essential disks
$D\subset V$ and $E\subset W$ such that $\partial
D\cap\partial E=\emptyset$; Otherwise, it is said to be strongly irreducible.

Let $S$ be a closed surface whose genus is at least 2. The distance between two essential simple closed curves
$\alpha$ and $\beta$ in $S$, denoted by $d_S(\alpha, \beta)$, is the
smallest integer $n \geq 0$ such that there is a sequence of
essential simple closed curves $\alpha=\alpha_{0},
\alpha_{1}...,\alpha_{n}=\beta$ in $S$ such that $\alpha_{i-1}$ is
disjoint from $\alpha_{i}$ for $1\leq i\leq n$. Let $A$ and $B$ be two sets of essential simple  closed curves in $S$. The distance between $A$ and $B$, which is denoted by $d_S(A,B)$, is defined to be
$min\{d_S(x,y)|x\in A, y\in B\}$.

If $D$ and $E$ are two compression disks on distinct sides of $S$, sometimes $d_S(\partial D,\partial E)$ is denoted simply by $d_S(D, E)$. For examples, suppose that $g(S)>1$ and $\partial D$ intersects $\partial E$ in one point then $d_S(D, E)=2$, since $\partial[ \overline{N(\partial D\cup\partial E})]\subset S$ is essential  and  disjoint from $\partial D\cup\partial E$.

The distance of the Heegaard splitting $V\cup_{S}W$ is defined to be $d(S)=d_S(\mathscr{D}_V,\mathscr{D}_W)$ where $\mathscr{D}_V$ and $\mathscr{D}_W$ are sets of essential disks in $V$ and $W$, respectively.
$d(S)$ was first defined by Hempel, see \cite{JH}.

$M=V\cup_{S}W$ is said to be a keen Heegaard splitting, if its distance is realized by a unique pair of elements $\mathscr {D}_V$ and $\mathscr {D}_W$.
For any integers $n\geq 2$ and $g\geqslant 3$, there exists a strongly keen Heegaard splitting of genus $g$ whose distance is $n$. For more details, see \cite{AY}.

The distance of $M=V\cup_{S}W$ is equal to 1 if and only if it is irreducible and weakly reducible.  By definition, $M=V\cup_{S}W$ is said to be a keen weakly reducible Heegaard splitting, if there exists only one pair of disjoint compression disks on opposite sides of the Heegaard surface.

The simplest example of keen weakly reducible Heegaard splittings is the genus 1 Heegaard splitting of $S^2\times S^1$ which is  reducible. Moreover, it is well known that there is no keen weakly reducible Heegaard splitting of genus 2. In the following statement, we assume that the the genus of the Heegaard surface is at least 3.

Suppose that $M=V\cup_{S}W$ is a keen weakly reducible Heegaard splitting, $D_0\subset V$ and $E_0\subset W$ are disjoint compression disks. It is easy to observe that $\partial D_0$ and  $\partial E_0$ are not isotopic, moreover,  each is non-separating on $S$, and $\partial D_0\cup\partial E_0$  is separating on $S$. It follows that $V\cup_{S}W$ is irreducible and unstabilized.

Let $F$ be a properly embedded, separating surface with no torus
components in a compact, orientable, irreducible 3-manifold $M$. Then the
disk complex of $F$, denoted by $\mathscr{D}(F)$, is defined as follows: Vertices of $\mathscr{D}(F)$ are isotopy classes of compression
disks for $F$, and a set of $m+1$ vertices forms an $m-$simplex if there are representatives for each that are pairwise disjoint.


David Bachman explored the information which is contained in the
topology of $\mathscr{D}(F)$ by defining the topological index
of $F$\cite{GT}. If $\mathscr{D}(F)$ is non-empty then the topological index of
$F$ is the smallest $n$ such that $\pi_{n-1}(\mathscr{D}(F))$ is
non-trivial. If $\mathscr{D}(F)$ is empty then $F$ will have
topological index $0$. If $F$ has a well-defined topological index
(i.e. $\mathscr{D}(F)=\emptyset$ or non-contractible) then we will say
that $F$ is a topologically minimal surface.

This raises a question that which surfaces are topologically
minimal surfaces, moreover, if a surface is topologically minimal,
we may ask what the topological index is. By definition, $F$ has
topological index 0 if and only if it is incompressible, and has
topological index 1 if and only if it is strongly irreducible. So
we only need to consider weakly reducible surfaces. If a weakly
reducible surface is topologically minimal, the topological index
is at least 2. Index 2 topologically minimal surface are called
critical surfaces which are also defined by David Bachman,
see\cite{DB}and\cite{DC}.

If a weakly reducible Heegaard splitting is keen, the disk complex of the Heegaard surface is quite simple.
By McCullough in\cite{DM}, who showed that the disk
complex of the boundary of a handlebody is contractible, the
 disk complex of the Heegaard surface is obtained by attaching these two
together with a single edge. That will just create a larger
contractible complex. Thus we have the following proposition.

\begin{prop}
 Suppose that $V\cup_{S}W$ is a keen weakly reducible Heegaard splitting, then the disk complex of the Heegaard surface $S$ is contractible, in other words, $S$ is not topologically minimal.
\end{prop}

In \cite{QL}, The author proved that some self-amalgamated Heegaard surfaces are keen. We remark that the genera of all the Heegaard surfaces given in \cite{QL} are odd, and the main theorem contains strict assumptions. In this paper, we will provide an explicit construction of keen weakly reducible Heegaard splittings of arbitrary genus. Such splittings admit a ``symmetrical structure'' and there is a lot of flexibility in doing so. Moreover, we may construct critical Heegaard splittings if we change some conditions.

\section{Preliminaries}

Let $V^*$ be a handlebody. Denote its boundary by $S^*$. Let $V$ be the handlebody obtained by attaching one 1-handle $D_0\times I$ to $V^*$ along a pair of disjoint disks $D_1, D_2\subset S^*$, where $D_0$ is a disk corresponding to the 1-handle. Thus $g(V)=g(V^*)+1$. Recall $\mathscr{D}_V$ is defined to be the set of compression disks of $V$. There is a natural partition of $\mathscr{D}_V$ induced by $D_0$ up to isotopy, as follows: $$\mathscr{D}_V=\mathscr{D}_0\cup\mathscr{D}_1\cup\mathscr{D}_2\cup\mathscr{D}_3$$

$\mathscr{D}_0=\{D_0\}$.

$\mathscr{D}_1=\{D~| D\cap D_0=\emptyset;~ D\neq D_0; D$ is inessential in $V^*\}$.

$\mathscr{D}_2=\{D~| D\cap D_0=\emptyset;~ D$ is essential in $V^*\}$.

$\mathscr{D}_3=\{D~| D\cap D_0\neq\emptyset\}$.

Let $D$ be a compression disk of $V$. Then $D$ belongs to one and only one of the four subsets.  If $D\subset \mathscr{D}_1$, then  $\partial D,~\partial D_1$ and $\partial D_2$ co-bound a pair of pants on $S^*$. It follows that $D$ is isotopic to a band-sum of $D_1, D_2$ along an arc and $D$ cuts $V$ into $V^*$ and a solid torus. $D\subset \mathscr{D}_2$ if and only if $D\subset \mathscr{D}_{V^*}$, obviously. See figure \ref{d3}.

Now consider $D\subset \mathscr{D}_3$. Suppose $D$ is an essential disk in $V$ such that $D\cap D_0\neq\emptyset$.  Furthermore, $D$ is isotoped in $V$ such that $|D\cap D_0|$ is minimal. Let $S_1$ be the surface  $S^*-(int(D_1)\cup int(D_2))$. Then $S_1$ is a sub-surface of $S^*$ with
two boundary components $\partial D_1$ and $\partial D_2$. By standard arguments, we have some observations as follows. See also\cite{ZY}.

\begin{lem}\label{1}
\mbox\par
\begin{enumerate} [(1)]
\item   Each component of $D\cap D_0$ is a properly embedded arc in both $D$ and $D_0$.
\item   Each component of $\partial D\cap S_1$ is essential on $S_1$.
\item   Each component of $D\cap(\partial D_0\times  I)$ is an arc with its two end points lying in distinct boundary components of the annulus $\partial D_0\times  I$.
\end{enumerate}
\end{lem}

 An essential arc $\gamma$ in $S_1$ is called strongly
essential if both boundary points lie in $\partial D_i$ and $\gamma$ is an essential arc on $S_1\cup D_j$ ,
where $\{i, j\} =\{1, 2\}$. Recall that $D\subset \mathscr{D}_3$. Let $\gamma$ be an outermost component of $D\cap(D_1\cup D_2)$ on $D$. This means that $\gamma$, together with an arc $\gamma_1\subset\partial D$, bounds a sub-disk in $D$, say $D_\gamma$, such that $D_\gamma$ intersects $D_\gamma\cap (D_1\cup D_2)=\gamma$. We call $\gamma_1$ an outermost arc related to $\gamma$ and call $D_\gamma$ an outermost disk related to $\gamma$.

\begin{lem}\label{2}\cite{ZY}

\mbox\par

\begin{enumerate} [(1)]
\item $\gamma_1$,  whose end points lie in one of $D_1$ and $D_2$, is strongly essential in $S_1$.
\item $D_\gamma\subset V^*$ and is essential in $V^*$.
\end{enumerate}

\end{lem}

\begin{figure}
\centering
    \includegraphics[width=7cm]{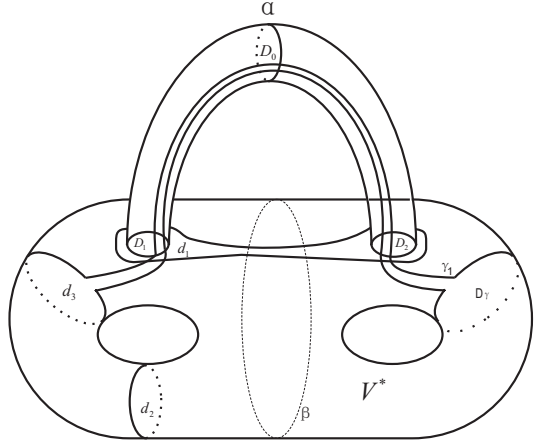}\\
  \caption{disks in the handlebody such that $d_i\subset \mathscr{D}_i$}\label{d3}
\end{figure}

\section{Construction of keen weakly reducible Heegaard splittings}

In this section, we will provide an explicit construction of keen weakly reducible Heegaard splittings of genus g.

Let $V^*$ be a genus $(g-1)$ handlebody. We choose an essential curve  $\beta\subset \partial V^*$, such that $\beta$ is separating on $\partial V^*$ and $d_{\partial V^*}(\beta, \mathscr{D}_{V^*})\geq 4$ .
Then we attach a 2-handle $E_0\times I$ to $V^*$ along $\beta$ so that $\beta$ bounds a disk $E_0$. This yields a 3-manifold denoted by $M_1$.

Next, we attach a 1-handle $D_0\times I$ to $M_1$ such that the gluing disks denoted by $D_1, D_2$ are on distinct components of $\partial V^*-\beta$, where $D_0$ is an disk corresponding to the 1-handle and we denote $\partial D_0$ by $\alpha$. The resulted 3-manifold denoted by $M_2$.  $\partial M_2$ is connected and $g(\partial M_2)=g-1$.

Next, let $W^*$ be another genus $(g-1)$ handlebody.  We glue $W^*$ to $M_2$ via an orientation preserving homeomorphism $f:\partial W^*\rightarrow \partial M_2$ such that
$d_{\partial W^*}(f(\alpha), \mathscr{D}_{W^*})\geq 4$. Since $\alpha$ is separating on $\partial M_2$, $f(\alpha)$ is separating on $\partial W^*$. Recall $D_1\cup D_2=V^*\cap (D_0\times I)$. Similarly,  $W^*\cap (E_0\times I)$ are two disks, denoted by $E_1$,$E_2$.

The result is a closed 3-manifold denoted by $M$. Denote $V=V^*\cup (D_0\times I)$, $W=W^*\cup (E_0\times I)$ and $S=V\cap W$. Then both $V$ and $W$ are genus $g$ handlebodies and $M=V\cup_SW$ is a Heegaard splitting.  $\alpha$ bounds the disk $D_0$ in $V$ and $\beta$ bounds the disk $E_0$ in $W$. Since the 2-handle $E_0\times I$ is attached before the 1-handle $D_0\times I$, $D_0\cap E_0=\alpha\cap\beta=\emptyset$. It follows that $M=V\cup_SW$ is weakly reducible.

\begin{thm}
 $M=V\cup_SW$ is a keen weakly reducible Heegaard splitting.
\end{thm}
\begin{proof}

In order to show $V\cup_SW$ is  keen, we firstly divide $\mathscr{D}_V$ and $\mathscr{D}_W$ as mentioned in Section 2:

$$\mathscr{D}_V=\mathscr{D}_0\cup\mathscr{D}_1\cup\mathscr{D}_2\cup\mathscr{D}_3$$ and
$$\mathscr{D}_W=\mathscr{D}^0\cup\mathscr{D}^1\cup\mathscr{D}^2\cup\mathscr{D}^3$$
such that:

$\mathscr{D}_0=\{D_0\}$.

$\mathscr{D}_1=\{D~| D\cap D_0=\emptyset;~ D\neq D_0; D$ is inessential in $V^*\}$.

$\mathscr{D}_2=\{D~| D\cap D_0=\emptyset;~ D$ is essential in $V^*\}= \mathscr{D}_{V^*}$.

$\mathscr{D}_3=\{D~| D\cap D_0\neq\emptyset\}$.

$\mathscr{D}^0=\{E_0\}$.

$\mathscr{D}^1=\{E~| E\cap E_0=\emptyset;~ E\neq E_0; D$ is inessential in $W^*\}$.

$\mathscr{D}^2=\{E~| E\cap E_0=\emptyset;~ E$ is essential in $W^*\}= \mathscr{D}_{W^*}$.

$\mathscr{D}^3=\{E~| E\cap E_0\neq\emptyset\}$.

In order to show that ($D_0$, $E_0$) is the unique disjoint pair of compression disks on distinct sides of $S$, we will show that for each $D\subset\mathscr{D}_i$, $E\subset\mathscr{D}^j$, where $i,j=0,1,2,3$ and $(i,j)\neq(0,0)$,  $D\cap E=\partial D\cap \partial E\neq\emptyset$ holds.

\begin{claim}\label{01}
For each $D\subset\mathscr{D}_0$, $E\subset\mathscr{D}^1$, $D\cap E\neq\emptyset$.
\end{claim}

In this case, $D=D_0$ and $E$ is a band-sum of two copies of $E_0$ along an arc. Recall that $\alpha=\partial D_0$ and $\beta=\partial E_0$ and the union separates $S$. $\alpha$ separates the two copies of $E_0$. Any arc connected to  the two copies intersects $\alpha$. It follows that $D\cap E\neq\emptyset$.

\begin{claim}\label{02}
For each $D\subset\mathscr{D}_0$, $E\subset\mathscr{D}^2$, $D\cap E\neq\emptyset$.
\end{claim}

In this case, $D=D_0$ and $E$ is essential both in handlebodies $W$ and $W^*$. Notice that $W^*=\overline{W\backslash (E_0\times I)}$. Recall that
$d(f(\alpha),\mathscr{D}_{W^*} )\geq 3$, thus $D_0$ intersects each compression disk of $W^*$.

\begin{rem}\label{r}
By the above two claims, if $E$ is an essential disk in $W$ which is not isotopic to $E_0$ and disjoint from $E_0$, then $D_0\cap E\neq\emptyset$.
Similarly, if $D$ is an essential disk in $V$ which is not isotopic to $D_0$ and disjoint from $D_0$, then $D\cap E_0\neq\emptyset$.

In the following argument, we assume that $|D\cap D_0|+|E\cap E_0|$ is minimal in the isotopy classes of $D$ and $E$.
\end{rem}

\begin{claim}\label{03}
For each $D\subset\mathscr{D}_0$, $E\subset\mathscr{D}^3$, $D\cap E\neq\emptyset$.
\end{claim}

In this case, $D=D_0$ and $E\cap E_0\neq\emptyset$. Assume that $D_0\cap E=\emptyset$. By applying Lemma \ref{2} to $W$ and $W^*$, there is an outermost disk of $E$, say $E_\gamma$, which is essential in $W^*$.  By Claim \ref{02}, $D_0\cap E_\gamma\neq\emptyset$. Notice that $D_0\cap E_\gamma\subset D_0\cap E$.  This contradicts the assumption $D_0\cap E=\emptyset$.

\begin{claim}\label{11}
For each $D\subset\mathscr{D}_1$, $E\subset\mathscr{D}^1$, $D\cap E\neq\emptyset$.
\end{claim}

In this case, $D$ is  a band-sum of two copies of $D_0$ and $E$ is a band-sum of two copies of $E_0$. $\partial D$ bounds a once-punctured torus $T_D$ on $S$ and $\partial E$ bounds a once punctured torus $T_E$. Suppose to the contrary that $D\cap E=\emptyset$. First we assume that $\partial E$ is isotopic to $\partial D$. In this case, since $g(S)>2$, we have $T_D=T_E$.  It follows that $\alpha$ and $\beta$ are isotopic because $\alpha\cap\beta=\emptyset$, a contradiction. Hence $\partial E\neq\partial D$. In this case $T_D\cap T_E=\emptyset$ and $\alpha\cup \beta$ does not separate the Heegaard surface, a contradiction.

\begin{claim}\label{12}
For each $D\subset\mathscr{D}_1$, $E\subset\mathscr{D}^2$, $D\cap E\neq\emptyset$.
\end{claim}

In this case, $D$ is  a band-sum of two copies of $D_0$ and $E$ is an essential disk in $W^*$. $\partial D$ bounds a once-punctured torus $T_D$ on $S$ and $\alpha\subset T_D$. Suppose to the contrary that $D\cap E=\emptyset$. If $\partial E\nsubseteq T_D$ then $E\cap D_0=\emptyset$ which contradicts Claim \ref{02}. Thus $\partial E\subset T_D\subset S$. Moreover, $\partial E\subset N(\alpha)\subset S$ since $E\cap E_0=\emptyset$. In this case $\partial E$ is isotopic to $\alpha$ which implies that $d_{\partial W^*}(f(\alpha),~\mathscr{D}_{W^*} )=0$, a contradiction.

\begin{claim}\label{13}
For each $D\subset\mathscr{D}_1$, $E\subset\mathscr{D}^3$, $D\cap E\neq\emptyset$.
\end{claim}

In this case $E\cap E_0\neq\emptyset$. By applying Lemma \ref{2} to $W$ and $W^*$, there is an outermost arc, say $\gamma$, whose two endpoints lie in one of $\partial E_1$ and $\partial E_2$. Furthermore, there is an outermost disk of $E$, say $E_{\gamma}$, such that $E_{\gamma}$ is essential in $W^*$. Without loss of generality, we assume that  $E_{\gamma}\cap E_1$ is an arc denoted by $e_1$, that is to say, $\partial\gamma\subset\partial E_1$.

$D$ is a band-sum of two copies of $D_0$ and the two copies lie distinct sides of $\beta$. Hence there is a sub-arc $\gamma_1$ of $\partial D$ such that $\partial\gamma_1\subset\partial E_2$.  $\gamma_1$, together with a sub-arc of $\partial E_2$, say $e_2$, bounds a disk isotopic to $D_0$.

$D\cap E=\emptyset$ implies that $\gamma\cap\gamma_1=\emptyset$,  thus $(\gamma\cup e_1)\cap(\gamma_1\cup e_2)=\emptyset$ which means that $E_{\gamma}\cap D_0=\emptyset$. This contradicts Claim \ref{02}.

\begin{claim}\label{22}
For each $D\subset\mathscr{D}_2$, $E\subset\mathscr{D}^2$, $D\cap E\neq\emptyset$.
\end{claim}

In this case, $D$ is essential in $V^*$ and $E$ is essential in $W^*$. Suppose to the contrary that $D\cap E=\emptyset$.
We remark that $\partial D\cap\beta\neq\emptyset$ and $\beta$ separates $\partial V^*$. Hence there is a sub-arc $\gamma_1$ of $\partial D$ such that $\partial\gamma_1\subset\partial E_1$ and $\gamma_1$ together with a arc of $E_1$ forms an essential closed curve $\gamma$ on $\partial W^*$.  Moreover, $D_0\cap D=\emptyset$ and $D_0\cap\beta=\emptyset$ mean that $D_0\cap\gamma=\emptyset$. $E\cap D=\emptyset$ and $E\cap\beta=\emptyset$ mean that $E\cap\gamma=\emptyset$.
Hence
 $d_{\partial W^*}(f(\alpha),\mathscr{D}_W^*)\leq d_{\partial W^*}(D_0,E)\leq d_{\partial W^*}(D_0,\gamma)+d_{\partial W^*}(\gamma, E)=1+1=2$, a contradiction.

\begin{claim}\label{23}
For each $D\subset\mathscr{D}_2$, $E\subset\mathscr{D}^3$, $D\cap E\neq\emptyset$.
\end{claim}

In this case, $D$ is essential in $V^*$ and $E\cap E_0\neq\emptyset$.  Suppose to the contrary that $D\cap E=\emptyset$.

By Lemma \ref{2}, there is an outermost arc $\gamma_2$ of $\partial E$ and an outermost disk of $E$ , say $E_{\gamma_2}$ which is essential in $W^*$.
Without loss of generality, we may assume that $\partial\gamma_2\subset\partial E_2$.

Since $D\subset\mathscr{D}_2$, by the argument mentioned in Claim\ref{22}, there is a sub-arc $\gamma_1$ of $\partial D$ such that $\partial\gamma_1\subset\partial E_1$ and $\gamma_1$ together with an arc of $E_1$ forms an essential closed curve $\gamma$ on $\partial W^*$ such that $D_0\cap\gamma=\emptyset$.

$D\cap E=\emptyset$ implies that $\gamma_1\cap\gamma_2=\emptyset$. According to our choice, $\partial\gamma_1$ and $\partial\gamma_2$ lie on distinct disks.
Hence $E_{\gamma_2}\cap\gamma=\emptyset$. Thus

$$d_{\partial W^*}(f(\alpha),\mathscr{D}_W^*)\leq d_{\partial W^*}(D_0,E_{\gamma_2}) \leq d_{\partial W^*}(D_0,\gamma)+d_{\partial W^*}(\gamma, E_{\gamma_2})=2$$

a contradiction.

\begin{claim}
For each $D\subset\mathscr{D}_3$, $E\subset\mathscr{D}^3$, $D\cap E\neq\emptyset$.
\end{claim}

In this case, $D\cap D_0\neq\emptyset$ and $E\cap E_0\neq\emptyset$.  Suppose to the contrary that $D\cap E=\emptyset$.

By Lemma \ref{2}, there exists an outermost arc $\gamma_1$ of $\partial D$ and an outermost disk of $D$ , say $D_{\gamma_1}$ which is essential in $V^*$. Since $D_{\gamma_1}\cap\beta\neq\emptyset$ and $\beta$ separates $\partial V^*$, there is an outermost arc say of $\gamma_{11}$ such that
$\gamma_{11}\subset\gamma_1$, moreover, $\partial\gamma_{11}\subset\partial E_1$ or $\partial\gamma_{11}\subset\partial E_2$. See Figure \ref{ot}.

Without loss of generality, we may assume that $\partial\gamma_{11}\subset\partial E_1$. Thus $\gamma_{11}$ together with an arc of $E_1$ forms an closed curve, say $\gamma$, which is essential in $\partial W^*$. We remark that $\gamma=\gamma_{11}\cup e_1$, where $e_1$ is an arc in $E_1$.
Since the outermost arc of $D$ and $E_2$ are disjoint from $\alpha$, $D_0\cap\gamma=\emptyset$.

By applying Lemma \ref{2} to $W$ and $W^*$, there is an outermost arc of $\partial E$, say $\gamma_2$, where $\partial \gamma_2\subset \partial E_1$ or $\partial \gamma_2\subset \partial E_2$.

$D\cap E=\emptyset$ means that $\gamma_{11}$ and $\gamma_2$ are disjoint.

If $\partial \gamma_2\subset \partial E_2$, $\gamma_2$, together with an arc in $E_2$, bounds an outermost disk, say $E_{\gamma_2}$, which is essential in $\partial W^*$.  In this case $\gamma$ and $E_{\gamma_2}$ are disjoint. It follows that:

$$d_{\partial W^*}(f(\alpha),\mathscr{D}_W^*)\leq d_{\partial W^*}(D_0,E_{\gamma_2}) \leq d_{\partial W^*}(D_0,\gamma)+d_{\partial W^*}(\gamma, E_{\gamma_2}) =2$$ which is a contradiction.

If $\partial \gamma_2\subset \partial E_1$,  $\gamma_2$, together with an arc $e_2$ in $E_1$, bounds an outermost disk, say $E_{\gamma_2}$, which is essential in $\partial W^*$.  In this case $\gamma$ intersects $E_{\gamma_2}$ in at most one point, since $e_1$  intersects  $e_2$ at most one point.
\begin{figure}
\centering
    \includegraphics[width=10cm]{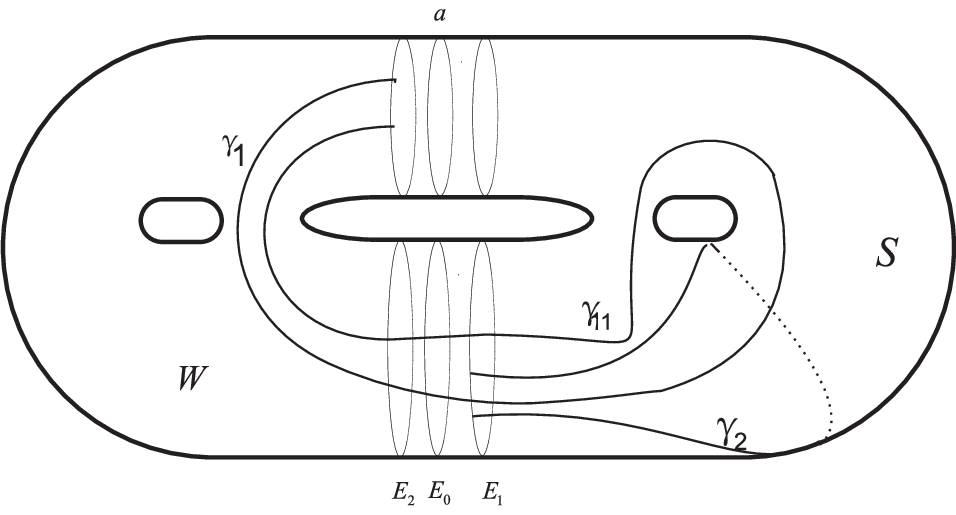}\\
  \caption{outermost arcs in $S\backslash N(\beta)$}\label{ot}
\end{figure}
It follows that:

$$d_{\partial W^*}(f(\alpha),\mathscr{D}_W^*)\leq d_{\partial W^*}(D_0,E_{\gamma_2}) \leq d_{\partial W^*}(D_0,\gamma)+d_{\partial W^*}(\gamma, E_{\gamma_2}) \leq 1+2=3$$

which is also a contradiction.

If for some $(i,j)$, each $D\subset\mathscr{D}_i$, $E\subset\mathscr{D}^j$, $D\cap E=\partial D\cap \partial E\neq\emptyset$ holds, then by the symmetric construction of the Heegaard splitting and the same partitions of the $\mathscr{D}_V$ and $\mathscr{D}_W$, it also holds for $(j,i)$.
This completes the proof.
\end{proof}

\section{on topologically minimal amalgamated Heegaard surfaces}

Let $N$ be a 3-manifold. Suppose $F$ is an incompressible separating surface which cuts $N$ into $N_1$ and $N_2$. Then $N$ is obtained by gluing two 3-manifolds, say $N_1$ and $N_2$, along a homeomorphism $f$. If $N_i$ admits a Heegaard splitting $V_i\cup_{S_i}W_i$ for $i=1,~2$, $N$ has a natural Heegaard splitting called the amalgamation of $V_i\cup_{S_i}W_i$ for $i=1,~2$, as:

$$N=C_1\cup_S C_2=(V_1\cup_{S_1}W_1)\cup_F(W_2\cup_{S_2}V_2)$$

 The amalgamation of two unstabilized Heegaard splittings may be stabilized, see\cite{TR}. However, if the gluing map $f$ is complicated enough, then the amalgamation of two minimal Heegaard splittings is unstabilized\cite{ML}. On the other hand, the amalgamation of two high distance Heegaard splittings is unstabilized\cite{TQ}.

David Bachman showed that if the gluing map $f$ is complicated enough, then the amalgamated Heegaard surface is not topologically minimal\cite{Db}.  We apply the idea of \cite{TQ} to conjecture that the amalgamated Heegaard surface of two high distance Heegaard splittings is not topologically minimal. Our construction of keen weakly reducible Heegaard splittings prsents positive examples.

\begin{thm}
Let $M$ be a 3-manifold. Suppose that $M$ admits an amalgamated Heegaard splitting:
$$M=V\cup_S W=(V_1\cup_{S_1}W_1)\cup_F(W_2\cup_{S_2}V_2).$$  such that $g(F)=g(S_1)=g(S_2)$  and $d(S_i)\geq 4$, for $i=1,~2$, then $M=V\cup_S W$ is keen weakly reducible hence $S$ is not topologically minimal.
\end{thm}

\begin{proof}
For each $i$, $d(S_i)\geq 4$ implies that $V_i\cup_{S_i}W_i$ is not an trivial Heegaard splitting. Since $g(\partial_+W_i)=g(S_i)=g(F)=g(\partial_-W_i)$, $F$ is disconnected and there exists only on separating compression disk embedded in $W_i$.(See figure \ref{gnn}) If $d(S_i)\geq 4$, it is easy to check that $V\cup_S W$ coincides with our construction.
\end{proof}

We remark that if $d(S_i)\geq 2$ then $S$ may not be keen and be topologically minimal. One example is that the standard Heegaard surface of $F^*\times S^1$, where $F^*\times S^1$ is a $S^1$ -bundle of a connected closed surface. In this case $V_i\cup_{S_i}W_i$ is the standard type 2 Heegaard splitting of $F^*\times I$ and $d(S_i)=2$.  the standard Heegaard surface of $F^*\times S^1$ is not keen and critical. In fact, if $d(S_i)\geq 2$ we show no higher index topologically minimal Heegaard surfaces exist.

\begin{figure}
\centering
    \includegraphics[width=5cm]{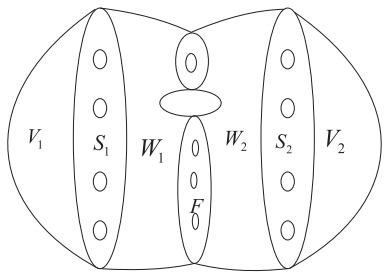}\\
  \caption{amalgamation of two Heegaard splittings}\label{gnn}
\end{figure}

\begin{thm}
Let $M$ be a 3-manifold. Suppose that $M$ admits an amalgamated Heegaard splitting:
$$M=V\cup_S W=(V_1\cup_{S_1}W_1)\cup_F(W_2\cup_{S_2}V_2).$$  such that $g(F)=g(S_1)=g(S_2)$  and $d(S_i)\geq 2$, for $i=1,~2$, then  $S$ is  either critical or not topologically minimal.
\end{thm}

\begin{proof} We use  notions as defined in Section 3.
 $d(S_i)\geq 2$ means that  $d_{\partial V^*}(\beta, \mathscr{D}_{V^*})\geq 2$ and $d_{\partial W^*}(f(\alpha), \mathscr{D}_{W^*})\geq 2$.
By Claim \ref{01}, Claim \ref{02} and Claim \ref{03}, if $d(S_i)\geq 2$, any disk in $\mathscr{D}_W\backslash E_0$ intersects $D_0$, equally,
any disk in $\mathscr{D}_V\backslash D_0$ intersects $E_0$. Hence if there exist $D\subset\mathscr{D}_V\backslash D_0$ and $E\subset\mathscr{D}_W\backslash E_0$, such that $D\cap E=\emptyset$ then $\pi_1(\mathscr{D}(S))\neq 1$ and $S$ is a critical Heegaard surface. if no such pair of disks exists, $S$ is keen weakly reducible and not topologically minimal.
\end{proof}


\bibliographystyle{amsplain}
\bibliographystyle{amsplain}

\end{document}